\documentclass[oneside,10pt]{article}          
\usepackage[b5paper]{geometry}	    
\usepackage{amsfonts,amsmath,latexsym,amssymb} 
\usepackage{theorem}                
\usepackage{mathrsfs,upref}         
\usepackage{mathptmx}		    
\usepackage{oam}	            
\newtheorem{theorem}{Theorem}
\newtheorem{lemma}{Lemma}
\newtheorem{corollary}{Corollary}

\theoremstyle{definition}

\newtheorem{definition}{Definition}
\newtheorem{example}{Example}

\begin{document}

\title[Matrices with totally positive powers]{Matrices with totally positive powers and their generalizations}


\author{Olga Y. Kushel}

\address{Olga Kushel, Institut f\"{u}r Mathematik, MA 4-2, Technische Universit\"{a}t Berlin, Strasse des 17. Juni 136  D-10623 Berlin, Germany}
\email{kushel@mail.ru}

\CorrespondingAuthor{Olga Y. Kushel}


\date{DD.MM.YYYY}                               



\thanks{The research leading to these results was carried out at Technische Universit\"{a}t Berlin and has received funding from the European Research Council under the European Union's Seventh Framework Programme (FP7/2007–2013) / ERC grant agreement n$^\circ$ 259173}

\begin{abstract}
        In this paper, eventually totally positive matrices (i.e. matrices all whose powers starting with some point are totally positive) are studied. We present a new approach to eventual total positivity which is based on the theory of eventually positive matrices. We mainly focus on the spectral properties of such matrices. We also study eventually J-sign-symmetric matrices and matrices, whose powers are $P$-matrices.
\end{abstract}

\maketitle



\section{Introduction}
In the 1940th Gantmacher and Krein described the spectral properties of strictly totally positive matrices (i.e. of those all whose minors are positive). Among this results they proved sufficient criteria for a matrix $\mathbf A$ to have a strictly totally positive power ${\mathbf A}^k$ for some positive even integer $k$ (see \cite{GANTKREI}).

Ever since, the theory of eventually positive matrices (i.e. matrices $\mathbf A$, whose all powers ${\mathbf A}^k$ are (entry-wise) positive starting with some positive integer $k_0$) was developed (see \cite{FRIED2}, \cite{TARH}, \cite{JT}). Such matrices were characterized by their spectral properties, the same as of positive matrices (i.e the biggest in absolute value eigenvalue is positive and the corresponding eigenvectors of both $\mathbf A$ and the transpose of $\mathbf A$ can be chosen to be positive). So the theory of positive matrices was extended to the matrices with some negative entries. Different aspects of eventual positivity were studied in \cite{BE}-\cite{EHT}, \cite{NO}. Total positivity of Hadamard powers of matrices as well as continuous real powers of totally positive matrices are studied in \cite{FAJ}.

In this paper, a new approach (through eventual positivity) to the matrices with totally positive powers, described by Gantmacher and Krein is provided. Using the theory of eventual positivity, we give necessary and sufficient characterization of eventually totally positive matrices (i.e. matrices whose all powers ${\mathbf A}^k$ are strictly totally positive starting from some positive integer $k_0$). Then, we analyze certain cone-theoretic generalization of strictly totally positive matrices and study some properties of eventually $P$-matrices (i.e. of matrices all whose powers starting with some point have positive principal minors).

The paper is organized as follows. We first collect definitions and statements on the Perron--Frobenius property and eventually positive matrices. We provide certain generalization of the Perron--Frobenius property which characterizes the class of matrices, eventually similar to positive ones. In section 2, we define eventually totally positive matrices and the equivalent property of matrix eigenvalues and eigenvectors. We provide some examples to show how the theory of totally positive matrices is extended to the matrices with some negative minors. Section 3 deals with similarity transformations preserving eventual total positivity and related properties. In section 4, we describe a cone-theoretic generalization of the class of eventually totally positive matrices, namely eventually totally J-sign-symmetric matrices. In section 5, we analyze the structure of eventually $P$-matrices with positive distinct spectra.

\section{Eventually positive matrices and their generalizations}
Here, as usual, we denote $\rho(A)$ the spectral radius of a matrix $\mathbf A$. An {\it eigenfunctional } of a matrix ${\mathbf A}$ corresponding to an eigenvalue $\lambda$ is defined as an eigenvector of ${\mathbf A}^T$ (the transpose of $\mathbf A$), corresponding to the same eigenvalue. Given a vector $x \in {\mathbb R}^n$ with the coordinates $(x^1, \ \ldots, \ x^n)$, we define the {\it signature vector} {\rm Sign}(x) as follows
$${\rm Sign}(x) := ({\rm sgn}(x^1), \ \ldots, \ {\rm sgn}(x^n))^T.$$
For a column vector $x = (x^1, \ \ldots, \ x^n)^T$ and a row vector $y = (y_1, \ \ldots, \ y_n)$, their tensor product $x \otimes y$ is considered as the following $n \times n$ matrix:
$$ x \otimes y = \begin{pmatrix} x_1y_1 & \ldots & x_1y_n \\
\ldots & \ldots & \ldots \\ x_ny_1 & \ldots & x_ny_n  \end{pmatrix}.$$

Let us recall some definitions (see, for example, \cite{ELS1}, \cite{ELS2}, \cite{JT}, \cite{NO}).

\begin{definition}
For a real $n \times n$ matrix $\mathbf A$, an eigenvalue $\lambda$ of $\mathbf A$ is called a
{\it dominant eigenvalue} if $|\lambda| = \rho(A)$. In addition, $\lambda$ is called a {\it strictly dominant eigenvalue}, if $|\lambda| > |\mu|$ for any other eigenvalue $\mu$ of $\mathbf A$.
\end{definition}

\begin{definition}
An $n \times n$ matrix $\mathbf A$ is said to have the {\it Perron--Frobenius property} if $\mathbf A$ has a positive dominant eigenvalue with the corresponding nonnegative eigenvector. A matrix
$\mathbf A$ is said to have the {\it strong Perron--Frobenius property} if $\mathbf A$ has
a unique positive simple strictly dominant eigenvalue with the corresponding positive eigenvector.
\end{definition}

Following Johnson and Tarazaga (see \cite{JT}), we denote $PF$ a class of all matrices $\mathbf A$ which have the strong Perron--Frobenius property together with their transposes ${\mathbf A}^T$.

Let us give the following generalization of the strong Perron--Frobenius property.

 \begin{definition}
 A matrix $\mathbf A$ is said to have the {\it signature equality property}
 if it satisfies the following conditions:
\begin{enumerate}
\item $\mathbf A$ has a unique positive simple strictly dominant eigenvalue $\lambda > 0$ with the corresponding eigenvector $x$ and eigenfunctional $x^*$;
\item Both $x$ and $x^*$ have no zero coordinates, and the equality ${\rm Sign}(x) = {\rm Sign}(x^*)$ holds.
\end{enumerate}
\end{definition}

If a matrix $\mathbf A$ has the strong Perron--Frobenius property, then it obviously has
the signature equality property. For the analogue of the Perron--Frobenius property, we have the following definition.

 \begin{definition}
 A matrix $\mathbf A$ is said to have the {\it weak signature equality property}
 if it satisfies the following conditions:
\begin{enumerate}
\item $\mathbf A$ has a positive dominant eigenvalue $\lambda > 0$ with the corresponding eigenvector $x$ and eigenfunctional $x^*$;
\item The inequalities $x^i(x^*)^i \geq 0$, $i = 1, \ \ldots, \ n$ hold for the coordinates $(x^1, \ \ldots, \ x^n)$ of the eigenvector $x$ and $((x^*)^1, \ \ldots, \ (x^*)^n)$ of the eigenfunctional $x^*$.
\end{enumerate}
\end{definition}

To show the link between the peripheral spectrum of a matrix and the asymptotic of matrix powers, we need the following lemma (the similar statement can be found in \cite{GANT}).

\begin{lemma}\label{L1} Let an $n \times n$ matrix ${\mathbf A}$ have a unique simple strictly dominant eigenvalue $\lambda_1 = \rho(A)$ with the corresponding eigenvector $x_1$ and eigenfunctional $x_1^*$. Then the following approximation holds:
$$\frac{1}{\rho(A)^k}{\mathbf A}^k \rightarrow x_1 \otimes x_1^* \qquad \mbox{as} \quad k \rightarrow \infty.$$
\end{lemma}
\begin{proof} The proof literally repeats the reasoning of Johnson and Tarazaga (see \cite{JT}, p. 328, proof of Theorem 1). Let us write the Jordan decomposition of $\mathbf A$:
$$ {\mathbf A} = {\mathbf S}{\Lambda}{\mathbf S}^{-1},$$
where ${\Lambda}$ is the Jordan canonical form of ${\mathbf A}$. In this case, we have
$${\Lambda} = \begin{pmatrix} \lambda_1 & 0  \\
 0 & \Lambda'\end{pmatrix} = \rho(A)\begin{pmatrix} 1 & 0 \\
0 & \frac{1}{\rho(A)}\Lambda' \end{pmatrix}.$$

The columns of the matrix $\mathbf S$ are the eigenvectors of $\mathbf A$ and the rows of of the matrix ${\mathbf S}^{-1}$ are the eigenfunctionals of $\mathbf A$. So the first column of the matrix $\mathbf S$ coincides with the first eigenvector $x_1$ and the first row of ${\mathbf S}^{-1}$ coincides with the first eigenfunctional $x_1^*$. Thus we can formally write ${\mathbf S} = \begin{pmatrix} x_1 & S' \end{pmatrix}$ and ${\mathbf S}^{-1} = \begin{pmatrix} x_1^* \\ (S^{-1})' \end{pmatrix}$. (Here $S'$ is an $n \times (n-1)$ matrix and $(S^{-1})'$ is an $(n-1)\times n$ matrix.) Since $\mathbf A$ and ${\mathbf A}^k$ share the same eigenvectors and eigenfunctionals, we obtain:
 $${\mathbf A}^k = {\mathbf S}{\Lambda}^k{\mathbf S}^{-1},  \qquad k = 2,3, \ldots. $$
Thus
$${\mathbf A}^k = \rho^k(A)\begin{pmatrix} x_1 & S' \end{pmatrix}\begin{pmatrix} 1 & 0 \\
0 & \frac{1}{\rho^k(A)}(\Lambda')^k \end{pmatrix}\begin{pmatrix}x_1^* \\ (S^{-1})'\end{pmatrix}.$$
Since $\rho(\frac{1}{\rho(A)}\Lambda') < 1$, we have $\frac{1}{\rho^k(A)}(\Lambda')^k \rightarrow 0$ as $k \rightarrow \infty$. In coordinates, it means
$$\frac{1}{\rho^k(A)}{\mathbf A}^k = $$ $$\begin{pmatrix} x_1^1 & s_{11}' & \ldots & s_{1n-1}' \\
x_1^2 & s_{21}' & \ldots & s_{2n-1}' \\
\ldots & \ldots & \ldots & \ldots\\
x_1^n & s_{nn-1}' & \ldots & s_{nn-1}'   \end{pmatrix}\begin{pmatrix} 1 & 0 & \ldots & 0 \\
0 & \epsilon_{11} & \ldots & \epsilon_{1n-1} \\
0 & \ldots & \ldots & \ldots \\
0 & \epsilon_{n-11} & \ldots & \epsilon_{n-1n-1} \end{pmatrix}\begin{pmatrix}(x_1^*)^1 & (x_1^*)^2 & \ldots & (x_1^*)^n \\ (s^{-1})_{11}' & (s^{-1})_{12}' & \ldots & (s^{-1})_{1n}' \\ \ldots & \ldots & \ldots & \ldots \\ (s^{-1})_{n-11}' & (s^{-1})_{n-12}' & \ldots & (s^{-1})_{n-1n}' \end{pmatrix}=$$
$$ \begin{pmatrix} x_1^1(x_1^*)^1  & x_1^1(x_1^*)^2  & \ldots & x_1^1(x_1^*)^n \\
x_1^2(x_1^*)^1  & x_1^2(x_1^*)^2  & \ldots & x_1^2(x_1^*)^n  \\
\ldots & \ldots & \ldots & \ldots \\
x_1^n(x_1^*)^1  & x_1^n(x_1^*)^2  & \ldots & x_1^n(x_1^*)^n  \\ \end{pmatrix} + \begin{pmatrix} \epsilon_{11}' & \epsilon_{12}' & \ldots &  \epsilon_{1n}' \\
\epsilon_{21}' &  \epsilon_{22}' & \ldots &  \epsilon_{2n}' \\
\ldots & \ldots & \ldots & \ldots \\
 \epsilon_{n1}' & \epsilon_{n2}' & \ldots & \epsilon_{nn}' \\ \end{pmatrix}.$$
Thus we obtain $\dfrac{1}{\rho^k}{\mathbf A}^k \rightarrow x_1\otimes x_1^*$ as $k \rightarrow \infty$.\qed
\end{proof}

Let us recall the following definition introduced in \cite{FRIED2}.

\begin{definition}
A real $n \times n$ real matrix ${\mathbf A}$ is called {\it eventually positive (EP)} (respectively, {\it eventually nonnegative (EN)}) if there exists a positive integer $k_0$ such that ${\mathbf A}^k > 0$ (respectively, ${\mathbf A}^k \geq 0$) for all $k \geq k_0$. For an EP matrix, the least such $k_0$ is called the {\it power index} of ${\mathbf A}$.
\end{definition}

The following statement is proved for eventually positive matrices (see \cite{JT}, p. 328, Theorem 1).

\begin{theorem}[Johnson, Tarazaga]
Let ${\mathbf A} \in {\mathbb R}^{n \times n}$. Then the following statements are equivalent.
\begin{enumerate}
\item[\rm 1.] Both of the matrices ${\mathbf A}$ and ${\mathbf A}^T$ have the strong Perron--Frobenius property.
\item[\rm 2.] The matrix ${\mathbf A}$ is eventually positive.
\item[\rm 3.] There is a positive integer $k$ such that ${\mathbf A}^k > 0$ and ${\mathbf A}^{k+1} > 0$.
\end{enumerate}
\end{theorem}

A weaker statement holds for eventually nonnegative matrices (see \cite{NO}, p. 136, Theorem 2.3).

\begin{theorem}[Noutsos]
Let ${\mathbf A} \in {\mathbb R}^{n \times n}$ be an eventually nonnegative matrix that is not nilpotent. Then both matrices ${\mathbf A}$ and ${\mathbf A}^T$ have the Perron--Frobenius property.
\end{theorem}

Note, that the converse statement may not hold (unlike the case of eventually positive matrices). For the counterexample, see \cite{ELS2}, p. 394, Example 2.5.

The following generalization of nonnegative matrices is introduced in \cite{KU}.

Let $J$ be any subset of $[n] := \{1, 2, \ldots, n\}$. Then $J^c:=[n] \setminus J$ and
$$[n] \times [n] = (J\times J)\cup(J^c \times J^c)\cup (J \times J^c) \cup (J^c \times J)$$ is a partition of $[n] \times [n]$ into four pairwise disjoint subsets.

\begin{definition}
 A matrix ${\mathbf A} = \{a_{ij}\}_{i,j=1}^n$ is called {\it J-sign-symmetric (JS)} if
$$a_{ij} \geq 0 \quad \mbox{on} \quad (J \times J)\cup (J^c \times J^c);$$
and
$$a_{ij} \leq 0 \quad  \mbox{on} \quad (J \times J^c)\cup (J^c \times J).$$
A matrix ${\mathbf A} = \{a_{ij}\}_{i,j=1}^n$ is called {\it strictly J-sign-symmetric (SJS)} if
$$a_{ij} > 0 \quad \mbox{on} \quad (J \times J)\cup (J^c \times J^c);$$
and
$$a_{ij} < 0 \quad  \mbox{on} \quad (J \times J^c)\cup (J^c \times J).$$
\end{definition}

Let us recall the following properties of JS matrices (see, for example, \cite{KU}).
\begin{enumerate}
\item[\rm 1.] A matrix ${\mathbf A}$ is JS (SJS) if and only if ${\mathbf A}$ can be represented  as follows:
\begin{equation}\label{1}{\mathbf A} = {\mathbf D} \widetilde{{\mathbf A}} {\mathbf
D}^{-1},
\end{equation}
 where $\widetilde{{\mathbf A}}$ is a
nonnegative (respectively, positive) matrix, ${\mathbf D}$ is a nonsingular diagonal matrix.

\item[\rm 2.] The spectral radius
$\rho(A)$ of an SJS matrix $\mathbf A$ is a simple positive eigenvalue of $\mathbf A$,
strictly bigger than the absolute value of any other eigenvalue of $\mathbf A$. The eigenvector $x_1$ and the eigenfunctional $x_1^*$, corresponding to the eigenvalue $\lambda_1 = \rho(A)$ may be chosen to satisfy the inequalities ${\mathbf D}x_1 > 0$, ${\mathbf D}x_1^* > 0$ entrywise, where $\mathbf D$ is an invertible diagonal matrix from Equality \eqref{1}.
\end{enumerate}
The proof of Property 2 immediately follows from the Perron theorem (see, for example, \cite{BERPL}, p. 27, Theorem 1.4).

Now let us examine a more general class of matrices, which includes eventually nonnegative ones.

\begin{definition}
A real $n \times n$ matrix ${\mathbf A}$ is called {\it eventually strictly J--sign-symmetric (ESJS)}
if there exists a positive integer $k_0$ such that ${\mathbf A}^k$ is SJS for all $k \geq k_0$.
A real $n \times n$ matrix ${\mathbf A}$ is called {\it eventually J--sign-symmetric (EJS)}
if there exists a positive integer $k_0$ such that ${\mathbf A}^k$ is JS for all $k \geq k_0$.
\end{definition}

Let us recall that the {\it sign pattern} ${\rm Sign}({\mathbf A})$ of a real $n \times n$ matrix ${\mathbf A} = \{a_{ij}\}_{i,j = 1}^n$ is defined by the equalities:
$${\rm Sign}({\mathbf A}) = \{s_{ij}\}_{i,j = 1}^n,$$
where $s_{ij} = {\rm sgn}(a_{ij})$ (see, for example, \cite{EHT}).

The following property of ESJS matrices holds.

\begin{lemma}
A matrix $\mathbf A$ is ESJS if and only if ${\mathbf A} = {\mathbf D} \widetilde{{\mathbf A}} {\mathbf
D}^{-1}$, where $\widetilde{{\mathbf A}}$ is an
EP matrix, ${\mathbf D}$ is a nonsingular diagonal matrix.
\end{lemma}

\begin{proof} $\Leftarrow$ Let ${\mathbf A} = {\mathbf D} \widetilde{{\mathbf A}} {\mathbf
D}^{-1}$, where $\widetilde{{\mathbf A}}$ is an EP matrix. Then there is a positive integer $k_0$ such that $(\widetilde{{\mathbf A}})^k$ is positive for all $k \geq k_0$. Examine the equalities
$${\mathbf A}^k = ({\mathbf D} \widetilde{{\mathbf A}} {\mathbf
D}^{-1})^k = {\mathbf D} (\widetilde{{\mathbf A}})^k {\mathbf
D}^{-1}.$$ Since $(\widetilde{{\mathbf A}})^k$ is positive, we apply Property 1 of SJS matrices and obtain that ${\mathbf A}^k$ is SJS for all $k \geq k_0$.

$\Rightarrow$ Let $\mathbf A$ be ESJS. Since ${\mathbf A}^{k_0}$ is SJS for some positive integer $k_0$, we have by Property 2 of SJS matrices that $\rho(A^{k_0})$ is a positive simple strictly dominant eigenvalue of ${\mathbf A}^{k_0}$. Since all the eigenvalues of ${\mathbf A}^{k_0}$ are powers of the eigenvalues of $\mathbf A$, we conclude that there is an eigenvalue $\lambda$ of $\mathbf A$ such that $\lambda^{k_0} = \rho(A^{k_0}) > 0$. Since $\rho(A^{k_0})$ is simple and strictly dominant we have that $\lambda$ is also simple and strictly dominant. Since $\mathbf A$ is real, we have that $\lambda$ is also real (positive or negative). Applying the same reasoning to ${\mathbf A}^{k_0 + 1}$ and taking into account that either $k_0$ or $k_0 + 1$ must be odd, we conclude that $\lambda > 0$. Thus the conditions of Lemma 1 hold. Applying Lemma 1, we obtain that $\frac{1}{\rho(A)^k}{\mathbf A}^k \rightarrow x_1 \otimes x_1^*$ as $k \rightarrow \infty$, where $x_1$ and $x_1^*$ are the eigenvector and the eigenfunctional corresponding to the simple strictly dominant eigenvalue $\lambda$. Thus there is a positive integer $k_1$ such that ${\rm Sign}({\mathbf A}^k) = {\rm Sign}(x_1 \otimes x_1^*)$ for all $k \geq k_1$.

Since ${\mathbf A}^k$ is SJS starting from $k = k_0$ and $\mathbf A$ and ${\mathbf A}^k$ share the same eigenvectors we have by Property 2 of SJS matrices that ${\mathbf D}x_1 >0$ and ${\mathbf D}x_1^* >0$ for some invertible diagonal matrix $\mathbf D$. This implies the equality ${\rm Sign}(x_1) = {\rm Sign}(x_1^*)$. Examine ${\mathbf D} = {\rm diag}[{\rm sgn}(x_1^1), \ {\rm sgn}(x_1^2), \ \ldots, \ {\rm sgn}(x_1^n)]$ and put $\widetilde{{\mathbf A}} = {\mathbf D}^{-1}{\mathbf A}{\mathbf D}$. In this case, it is easy to see that the matrix ${\mathbf D}^{-1}{\rm Sign}(x_1 \otimes x_1^*){\mathbf D}$ is positive. So we obtain
$${\rm Sign}({\mathbf D}^{-1}{\mathbf A}{\mathbf D})^k =  {\mathbf D}^{-1}{\rm Sign}({\mathbf A}^k){\mathbf D} = {\mathbf D}^{-1}{\rm Sign}(x_1 \otimes x_1^*){\mathbf D}$$ for $k \geq k_1$. This equality shows that $({\mathbf D}^{-1}{\mathbf A}{\mathbf D})^k$ is positive for all $k \geq  k_1$. Thus ${\mathbf A} = {\mathbf D} \widetilde{{\mathbf A}} {\mathbf
D}^{-1}$, where $\widetilde{{\mathbf A}} = {\mathbf D}^{-1}{\mathbf A}{\mathbf D}$ is an EP matrix. \qed
\end{proof}

The following statement generalizes results of Johnson and Tarazaga.
\begin{theorem}
 Let ${\mathbf A} \in {\mathbb R}^{n \times n}$. Then the following statements are equivalent.
\begin{enumerate}
\item[\rm 1.] The matrix ${\mathbf A}$ has the signature equality property.
\item[\rm 2.] The matrix ${\mathbf A}$ is eventually SJS.
\item[\rm 3.] There is a positive integer $k$ such that both ${\mathbf A}^k$ and ${\mathbf A}^{k+1}$ are SJS.
\end{enumerate}
\end{theorem}
\begin{proof}
${\rm (1)} \Rightarrow {\rm (2)}$
Suppose (1). Let $\mathbf A$ has a unique positive simple strictly dominant eigenvalue $\lambda_1 > 0$ with the corresponding eigenvector $x_1$ and eigenfunctional $x_1^*$. Let, in addition, both $x_1$ and $x_1^*$ have no zero coordinates, and the equality ${\rm Sign}(x_1) = {\rm Sign}(x_1^*)$ holds. Applying Lemma 1, we obtain that there is a positive integer $k_0$ such that $${\rm Sign}({\mathbf A}^k) = {\rm Sign}(x_1 \otimes x_1^*)$$ for all $k \geq k_0$.

Let us organize a partition of $[n]$ as follows. We put $J := \{i \in [n]: x_1^i > 0\}$. In this case,
$J^c = [n]\setminus J = \{i \in [n] : x_1^i < 0\}$. Since ${\rm sgn}(x_1^i) = {\rm sgn}((x_1^*)^i)$ for all $i = 1, \ \ldots, \ n$ and all of the coordinates are nonzero, we have
$$x_1^i(x_1^*)^j >0 \qquad \mbox{if} \quad i,j \in J \quad \mbox {or} \quad i,j \in J^c;$$
$$x_1^i(x_1^*)^j <0 \qquad \mbox{if} \quad i \in J, j \in J^c \quad \mbox{or} \quad i \in J^c, j \in J.$$
Thus the matrix $x_1 \otimes x_1^*$ is SJS according to Definition 6.

${\rm (2)} \Rightarrow {\rm (3)}$ Since $\mathbf A$ is eventually SJS, we obtain by Lemma 2 that ${\mathbf A} = {\mathbf D}\widetilde{{\mathbf A}}{\mathbf D}^{-1}$, where $\widetilde{{\mathbf A}}$ is an EP matrix, ${\mathbf D}$ is a nonsingular diagonal matrix. Thus there is a positive integer $k_0$ such that $\widetilde{{\mathbf A}}^k$ is positive for all $k \geq k_0$. Considering the equalities
$$ {\mathbf A}^{k_0} = {\mathbf D}\widetilde{{\mathbf A}}^{k_0}{\mathbf D}^{-1} $$
and
$$ {\mathbf A}^{k_0 + 1} = {\mathbf D}\widetilde{{\mathbf A}}^{k_0 + 1}{\mathbf D}^{-1} $$
we obtain that
$${\rm sgn}((a^{k_0})_{ij}) = {\rm sgn(d_{ii})}{\rm sgn}((\widetilde{a}^{k_0})_{ij}){\rm sgn}(\frac{1}{d_{jj}});$$
$${\rm sgn}((a^{k_0 + 1})_{ij}) = {\rm sgn(d_{ii})}{\rm sgn}((\widetilde{a}^{k_0 + 1})_{ij}){\rm sgn}(\frac{1}{d_{jj}}),$$
where $(a^k)_{ij}$, $(\widetilde{a}^k)_{ij}$ denote the entries of the matrices ${\mathbf A}^k$ and $\widetilde{{\mathbf A}}^k$, respectively, for $k = k_0, \ k_0 + 1$.
Since ${\rm sgn}((\widetilde{a}^{k_0})_{ij}) = {\rm sgn}((\widetilde{a}^{k_0 + 1})_{ij}) = + 1$ for all $i,j = 1, \ \ldots, \ n$, we have ${\rm Sign}({\mathbf A}^{k_0}) = {\rm Sign}({\mathbf A}^{k_0 + 1})$.

${\rm (3)} \Rightarrow {\rm (1)}$ Let both the matrices ${\mathbf A}^k$ and ${\mathbf A}^{k+1}$ are STJS for some positive integer $k$. Since ${\mathbf A}^k$ is SJS, $\rho(A^k)$ is an eigenvalue of ${\mathbf A}^k$ (by Property 2), and there must be an eigenvalue $\lambda$ of $\mathbf A$ such that
$\lambda^k = \rho(A^k)$. Since $\rho(A^k)$ is a positive simple strictly dominant eigenvalue of ${\mathbf A}^k$, we conclude that $\lambda$ is a real simple strictly dominant eigenvalue of $\mathbf A$ (positive or negative). Applying the same reasoning for ${\mathbf A}^{k+1}$ and taking into account that one of the integers $k$ and $k + 1$ must be odd, we conclude $\lambda > 0$.

 Applying again Property 2 to ${\mathbf A}^k$, we get that the corresponding to $\rho(A^k)$ eigenvector $x_1$ and eigenfunctional $x_1^*$ satisfy the equality ${\rm Sign}(x_1) = {\rm Sign}(x_1^*)$. Observing that $x_1$ and $x_1^*$ are also the eigenvector and the eigenfunctional of $\mathbf A$, corresponding to the eigenvalue $\lambda$ we complete the proof.
\qed
\end{proof}

\section{Eventually totally positive matrices}
Let us recall some definitions and notations of the exterior algebra.

\begin{definition}
Let $e_1, \ \ldots, \ e_n$ be an arbitrary basis in ${\mathbb R}^n$ and let $x_1, \ \ldots, \ x_j$ $ \ (2 \leq j \leq n)$ be any vectors in ${\mathbb R}^n$ defined by their coordinates: $x_i = (x_i^1, \ \ldots, \ x_i^n)$, $i = 1, \ \ldots, \ j$. Then the vector $x_1 \wedge \ldots \wedge x_j \in {\mathbb R}^{\binom{n}j}$ (here $\binom{n}j = \frac{n!}{j!(n-j)!}$) with coordinates of the form
$$(x_1 \wedge \ldots \wedge x_j)^{\alpha}:= \begin{vmatrix} x_1^{i_1} & \ldots & x_j^{i_1} \\
\ldots & \ldots & \ldots \\
x_1^{i_j} & \ldots & x_j^{i_j} \\
 \end{vmatrix},$$ where $\alpha$ is the number of the set of indices $(i_1, \ \ldots,  \ i_j) \subseteq [n]$ in the lexicographic ordering, is
called an {\it exterior product} of $x_1, \ \ldots, \
x_j$.
\end{definition}

We consider the $j$th exterior power $\wedge^j {\mathbb R}^n$ of the space ${\mathbb R}^n$ as the space ${\mathbb R}^{\binom{n}j}$. The set of all of exterior
products of the form $e_{i_1} \wedge \ldots \wedge e_{i_j}$, where
$1 \leq i_1 < \ldots < i_j \leq n$ forms a canonical basis in $\wedge^j {\mathbb R}^n$ (see \cite{GLALU}).

\begin{definition}
 Let $A: {\mathbb R}^n \rightarrow {\mathbb R}^n$ be a linear operator. Then its {\it $j$th exterior power} $\wedge^j A$ is defined as on operator on $\wedge^j
{\mathbb R}^n$ acting by the rule
$$ (\wedge^j A)(x_1 \wedge \ldots \wedge x_j) = Ax_1 \wedge \ldots \wedge Ax_j.$$
\end{definition}

If ${\mathbf A} =
\{a_{ij}\}_{i,j = 1}^n$ is the matrix of $A$ in a basis $e_1, \ \ldots, \ e_n$, then
the matrix of $\wedge^j A$ in the basis
$\{e_{i_1} \wedge \ldots \wedge e_{i_j}\}$, where $1 \leq i_1<
\ldots < i_j \leq n$, equals the $j$th compound matrix $
{\mathbf A}^{(j)}$ of the initial matrix ${\mathbf A}$. Here the $j$th compound matrix $
{\mathbf A}^{(j)}$ consists of all the minors of the $j$th order
$A\begin{pmatrix}
  i_1 &  \ldots & i_j \\
  k_1 & \ldots & k_j \end{pmatrix}$, where $1 \leq i_1<
\ldots < i_j \leq n, \ 1 \leq k_1< \ldots < k_j \leq n$, of the
initial $n \times n$ matrix ${\mathbf A}$, listed in the lexicographic order (see, for example,
\cite{PINK}).

It is easy to see that ${\mathbf A}^{(1)} = {\mathbf A}$ and ${\mathbf A}^{(n)}$ is one-dimensional and coincides with $\det ({\mathbf A})$.

The following properties of the compound matrices will be used later (see, for example, \cite{GANTKREI}).
\begin{enumerate}
\item[\rm 1.] Let ${\mathbf A}, \ {\mathbf B}$ be $n \times n$ matrices. Then $({\mathbf A}{\mathbf B})^{(j)} = {\mathbf A}^{(j)}{\mathbf B}^{(j)}$ for $j = 1, \ \ldots, \ n$ (the Cauchy--Binet formula).
\item[\rm 2.] The $j$-th compound of an invertible matrix is invertible and the following equality holds: $({\mathbf A}^{(j)})^{-1} = ({\mathbf A}^{-1})^{(j)}$, $j = 1, \ \ldots, \ n$ (the Jacobi formula).
\end{enumerate}

Let us recall the statement concerning the eigenvalues of the exterior
power of an operator (see, for example, \cite{PINK}, p. 132).

\begin{theorem}[Kronecker] Let $\{\lambda_{i}\}_{i = 1}^n$ be the set of
all eigenvalues of an $n \times n$ matrix $\mathbf A$, repeated according to multiplicity. Then all the possible products of the form $\{\lambda_{i_1} \ldots \lambda_{i_j} \}$, where $1
\leq i_1 < \ldots < i_j \leq n$, forms the set of all the possible eigenvalues
of the $j$th compound matrix ${\mathbf A}^{(j)}$, repeated
according to multiplicity. If $x_{i_1}, \ \ldots, \ x_{i_j}$ are linearly independent eigenvectors of $\mathbf A$, corresponding to the eigenvalues $\lambda_{i_1}, \ \ldots, \ \lambda_{i_j}$ $(1
\leq i_1 < \ldots < i_j \leq n)$ respectively, then their exterior product $x_{i_1} \wedge \ldots \wedge x_{i_j}$ is an eigenvector of ${\mathbf A}^{(j)}$, corresponding to the eigenvalue $\lambda_{i_1} \ldots \lambda_{i_j}$.
\end{theorem}

As usual, we denote $S^-(x)$ the number of sign changes in the sequence of the
coordinates $(x_1, \ \ldots, \ x_n)$ of the vector $x$ with zero
coordinates discarded, and $S^+(x)$ the  maximum number of sign changes in the sequence
$(x_1, \ \ldots, \ x_n)$, where zero coordinates are arbitrarily
assigned values $\pm 1$.

The following definition was given in \cite{GANTKREI}.
\begin{definition}
A system of nonzero vectors $\{x_1, \ \ldots, \ x_n\}$, $x_i \in {\mathbb R^n}$, is called a {\it Markov system} or {\it an oscillating system} if any linear combination $x = \sum\limits_{i=1}^jc_ix_i$ satisfies the inequality
$$S^+(x) \leq j-1, $$
whenever $1 \leq j \leq n$, $c_i \in {\mathbb R}$, $\sum\limits_{i=1}^jc_i^2 \neq 0$.
In particular, the $j$th vector $x_j$ in the oscillating system $\{x_1, \ \ldots, \ x_n\}$ has exactly $j-1$ changes of sign.
\end{definition}

\begin{definition}
A real $n \times n$ matrix ${\mathbf A}$ is said to have the {\it Gantmacher--Krein property} if
$\mathbf A$ has $n$ positive simple eigenvalues $\{\lambda_1, \ \ldots, \ \lambda_n\}$ with the
Markov systems of the corresponding eigenvectors $\{x_1, \ \ldots, \ x_n\}$.
\end{definition}

Let us denote GK the class of all matrices which possess the Gantmacher--Krein property together with their transposes.

The following lemma is proved in \cite{AN} (see \cite{AN}, p. 198, Lemma 5.1).

\begin{lemma}[Ando]
Let $\{x_1, \ \ldots, \ x_j\}$ be real vectors from ${\mathbb R}^n$ $( j < n)$. In order that $$S^+(\sum_{i=1}^jc_ix_i) \leq j-1, $$
whenever $c_i \in {\mathbb R}$, $\sum\limits_{i=1}^jc_i^2 \neq 0$, it is necessary and sufficient that
$x_1\wedge \ldots \wedge x_j$ be strictly positive or strictly negative.
\end{lemma}

Now we prove the following result.

\begin{lemma}
Let ${\mathbf A} \in {\mathbb R}^{n \times n}$. Then the following are equivalent.
\begin{enumerate}
\item[\rm 1.] The matrix ${\mathbf A}$ has the Gantmacher--Krein property.
\item[\rm 2.] The $j$th compound matrix ${\mathbf A}^{(j)}$ has the strong Perron--Frobenius property for all $j = 1, \ \ldots, \ n$.
\end{enumerate}
\end{lemma}
\begin{proof}
$(1) \Rightarrow (2).$ Assume that $\mathbf A$ has the Gantmakher--Krein property, i.e. $\mathbf A$ has $n$ positive simple eigenvalues $\{\lambda_1, \ \ldots, \ \lambda_n\}$, $\lambda_1 > \lambda_2 > \ldots > \lambda_n > 0$, and the corresponding eigenvectors $\{x_1, \ \ldots, \ x_n\}$ form a Markov system. Let us consider the $j$th compound matrix ${\mathbf A}^{(j)}$. Applying the Kronecker theorem (Theorem 4) to ${\mathbf A}^{(j)}$ we obtain that $\rho(A^{(j)}) = \lambda_1 \ldots \lambda_j$ is a positive simple strictly dominant eigenvalue of ${\mathbf A}^{(j)}$ with the corresponding eigenvector $x_1 \wedge \ldots \wedge x_j$. Since the eigenvectors $x_1, \ \ldots, \ x_n$ form a Markov system, any linear combination $x = \sum\limits_{i=1}^jc_ix_i$ satisfies the inequality
$$S^+(x) \leq j-1, $$
whenever $1 \leq j \leq n$, $c_i \in {\mathbb R}$, $\sum\limits_{i=1}^jc_i^2 \neq 0$. Thus, applying Ando's lemma (Lemma 3), we get that the eigenvector $x_1 \wedge \ldots \wedge x_j$ may be chosen to be strictly positive.

$(2) \Rightarrow (1).$ The proof of the statement that all the eigenvalues $\lambda_1, \ \ldots, \ \lambda_n$ of $\mathbf A$ are positive, simple and different in absolute value of each other literarily repeats the corresponding reasoning of Gantmakher and Krein (see \cite{GANTKREI}, or \cite{PINK}, p. 130, the proof of Theorem 5.3) Let us list the eigenvalues of the matrix $\mathbf A$
in descending order of their absolute values (taking into account their
multiplicities):
$$|\lambda_{1}| \geq | \lambda_{2}| \geq |\lambda_{3}| \geq \ldots
\geq |\lambda_{n}|.$$

Since the matrix ${\mathbf A}$ has the strong Perron--Frobenius property,
we get: $\lambda_{1} = \rho({\mathbf A})>0$ is a simple positive eigenvalue
of ${\mathbf A}$, different in absolute value from the remaining eigenvalues.
Examine the second compound matrix ${\mathbf A}^{(2)}$ which also has the strong Perron--Frobenius property we get: $\rho({\mathbf A}^{(2)}) > 0$ is
a simple positive eigenvalue of ${\mathbf A}^{(2)}$, different in absolute value from the remaining eigenvalues. As it follows from the statement of the Kronecker theorem, ${\mathbf A}^{(2)}$ has no other eigenvalues, except all the possible products of the form $\lambda_{i_1}\lambda_{i_2}$ where $1 \leq i_1 < i_2 \leq n$. Therefore $\rho({\mathbf A}^{(2)})>0$ can be represented in the
form of the product $\lambda_{i_1}\lambda_{i_2}$ with some values of
the indices $i_1,i_2$, \ $i_1 < i_2$. The facts that the
eigenvalues are listed in a descending order and there is only
one eigenvalue on the spectral circle $|\lambda| = \rho({\mathbf A})$ imply that
 $\rho({\mathbf A}^{(2)}) =  \lambda_{1}\lambda_{2} = \rho({\mathbf A})\lambda_{2}$.
  Therefore $\lambda_{2} = \frac{\rho({\mathbf A}^{(2)})}{\rho({\mathbf A})}>0$.

Repeating the same reasoning for ${\mathbf A}^{(j)}$, $j = 3, \ \ldots, \ n$, we obtain the relations:
$$\lambda_{j} = \frac{\rho({\mathbf A}^{(j)})}{\rho({\mathbf A}^{(j-1)})}>0,$$
where $j = 3, \ \ldots, \ n$. The simplicity of the eigenvalues $\lambda_j$ for every $j$ also follows from the above relations and the simplicity of $\rho({\mathbf A}^{(j)})$.

Then, applying Ando's lemma (Lemma 3) to all the exterior products of the form $x_1 \wedge \ldots \wedge x_j$ (they are all positive), we obtain that the eigenvectors $\{x_1, \ \ldots, \ x_n\}$ form a Markov system.
\qed
\end{proof}

Let us recall some well-known definitions (see, for example, \cite{GANT}).

\begin{definition}
A real $n \times n$ matrix ${\mathbf A}$ is called {\it totally positive (TP)} if ${\mathbf A}$ is
nonnegative and its $j$th compound matrix ${\mathbf A}^{(j)}$ is also nonnegative for
all $j = 2, \ \ldots, \ n$.

A real $n \times n$ matrix ${\mathbf A}$ is called {\it strictly totally positive (STP)} if
${\mathbf A}$ is positive and its $j$th compound matrix ${\mathbf A}^{(j)}$ is also positive for
all $j = 2, \ \ldots, \ n$.
\end{definition}

\begin{definition}
A real $n \times n$ matrix ${\mathbf A}$ is called {\it oscillatory} if it is TP and there is
a positive integer $k$ such that ${\mathbf A}^k$ is STP.
\end{definition}
Obviously, every oscillatory matrix is STP.

The following statement holds for STP matrices (see \cite{GANTKREI} or \cite{PINK}, p. 130, Theorem 5.3)
\begin{theorem}[Gantmacher, Krein]
Let an $n \times n$ matrix ${\mathbf A}$ be STP. Then all the
eigenvalues of ${\mathbf A}$ are positive and simple:
$$\rho(A)= \lambda_1 > \lambda_2 > \ldots > \lambda_n > 0.$$
 The first eigenvector corresponding to the maximal
eigenvalue $\lambda_1$ is strictly positive and the $j$th
eigenvector $x_j$ corresponding to the $j$th in absolute value
eigenvalue $\lambda_j$ has exactly $j - 1$ changes of sign. Moreover, the following inequalities hold:
$$q-1 \leq S^-(\sum_{i=q}^p c_i x_i) \leq S^+(\sum_{i=q}^p c_i x_i) \leq p-1$$
for each $1 \leq q \leq p \leq n$ and $\sum\limits_{i=q}^p c_i^2 \neq 0$.
\end{theorem}

Now let us introduce the following generalization of the class of STP matrices.
\begin{definition}
A real $n \times n$ matrix ${\mathbf A}$ is called {\it eventually strictly totally positive (ESTP)} if there is a positive integer $k_0$ such that for all $k \geq k_0$ ${\mathbf A}^k$ is STP. Here the minimal value of $k_0$ is called the {\it power index} of a ESTP matrix ${\mathbf A}$.
\end{definition}
It follows from the given above definition that {\it an oscillatory matrix is ESTP}. But the class of ESTP matrices also includes matrices with some negative entries and some negative minors.

Note that {\it if ${\mathbf A}$ and ${\mathbf B}$ are ESTP (ETP) matrices and ${\mathbf A}{\mathbf B} = {\mathbf B}{\mathbf A}$, then ${\mathbf A}{\mathbf B}$ is also ESTP (ETP).} However, if ${\mathbf A}{\mathbf B} \neq {\mathbf B}{\mathbf A}$ the above property may not hold. Moreover, ${\mathbf A}{\mathbf B}$ might be ESTP (ETP) yet neither ${\mathbf A}$ nor ${\mathbf B}$ is so.

The following theorem was proved in \cite{GANTKREI}.
\begin{theorem} Let an $n \times n$ matrix $\mathbf A$ have $n$ different in absolute value nonzero eigenvalues $\lambda_1, \ \ldots, \ \lambda_n$:
$$|\lambda_1| > |\lambda_2| > \ldots > |\lambda_n| > 0.$$
and the eigenvectors $x_1, \ \ldots, \ x_n$ and $x_1^*, \ \ldots, \ x_n^*$ of ${\mathbf A}$ and ${\mathbf A}^T$, respectively, form two Markov systems. Then there is a positive integer $k$ such that ${\mathbf A}^k$ is strictly totally positive.
\end{theorem}

The proof of Theorem 6 given by Gantmacher and Krein implies that in the case when some of the eigenvalues $\lambda_i$, $i = 1, \ \ldots, \ n$, are negative, the value $k$ is necessarily even.

The following corollary concerns oscillatory matrices (see \cite{GANTKREI}).
\begin{corollary} If a totally nonnegative matrix $\mathbf A$ satisfies the following conditions
\begin{enumerate}
\item[\rm 1.] all the eigenvalues of $\mathbf A$ are positive and simple;
\item[\rm 2.] all the eigenvectors of $\mathbf A$ and ${\mathbf A}^T$ forms two Markov systems,
\end{enumerate}
then $\mathbf A$ is oscillatory.
\end{corollary}

Now let us prove the main result, which characterizes properties of ESTP matrices.

\begin{theorem}
Let ${\mathbf A} \in {\mathbb R}^{n \times n}$. Then the following statements are equivalent.
\begin{enumerate}
\item[\rm 1.] Both of the matrices ${\mathbf A}$ and ${\mathbf A}^T$ have the Gantmacher--Krein property.
\item[\rm 2.] For every $j$, $j = 1, \ \ldots, \ n$, both the $j$th compound matrix ${\mathbf A}^{(j)}$ and its transpose $({\mathbf A}^{(j)})^T$ have the strong Perron--Frobenius property.
\item[\rm 3.] For every $j$, $j = 1, \ \ldots, \ n$, the $j$th compound matrix ${\mathbf A}^{(j)}$ is eventually positive.
\item[\rm 4.] The matrix ${\mathbf A}$ is eventually strictly totally positive.
\end{enumerate}
\end{theorem}

\begin{proof} $(1) \Rightarrow (2).$ Applying Lemma 4 to the matrix $\mathbf A$, we obtain that the $j$th compound matrix ${\mathbf A}^{(j)}$ has the strong Perron--Frobenius property for every $j$, $j = 1, \ \ldots, \ n$. Applying Lemma 4 to ${\mathbf A}^T$ we obtain that $({\mathbf A}^T)^{(j)}$ has the strong Perron--Frobenius property for every $j$, $j = 1, \ \ldots, \ n$. Observing that $({\mathbf A}^T)^{(j)} = ({\mathbf A}^{(j)})^T$ for every $j$, $j = 1, \ \ldots, \ n$, we complete the proof.

$(2) \Rightarrow (3).$ It is enough for the proof, just to apply Theorem 1.

$(3) \Rightarrow (4).$ Since the compound matrices ${\mathbf A}^{(j)}$ are eventually positive for
all $j = 1, \ \ldots, \ n$, we can find the power index $k_j$ such that $({\mathbf A}^{(j)})^k$ is positive for all positive integers $k \geq k_j$. Then let us fix $k_0 = \max\limits_j(k_j)$ and examine ${\mathbf A}^k$, for $k \geq k_0$. Applying the Cauchy--Binet formula, we obtain that
$$({\mathbf A}^k)^{(j)} = ({\mathbf A}^{(j)})^{k},$$ for all $j = 1, \ \ldots, \ n$, and $({\mathbf A}^{(j)})^{k}$ is positive since $k \geq k_0 \geq k_j$. Thus ${\mathbf A}^k$ is STP for all $k \geq k_0$ and $\mathbf A$ is ESTP.

$(4) \Rightarrow (1).$ Since the matrix ${\mathbf A}$ is eventually strictly totally positive, we can find a power index $k_0$ such that ${\mathbf A}^k$ is STP for all $k \geq k_0$. Applying Theorem 5 to ${\mathbf A}^k$, we obtain that all the eigenvalues $\mu_1, \ \ldots, \ \mu_n$ of ${\mathbf A}^k$ are positive, simple and strictly different in absolute value from each other:
$$\mu_1 > \mu_2 > \ldots > \mu_n > 0.$$ The corresponding eigenvectors $(x_1, \ \ldots, \ x_n)$ form a Markov system. Since all eigenvalues of ${\mathbf A}^k$ are just powers of the eigenvalues of $\mathbf A$, there is an eigenvalue $\lambda_i$ of $\mathbf A$ such that $\lambda_i^k = \mu_i$, $i = 1, \ \ldots, \ n$. Thus we obtain that all the eigenvalues of $\mathbf A$ are simple, real (positive or negative) and different in absolute value from each other. Applying the same reasoning to ${\mathbf A}^{k+1}$ and observing that either $k$ or $k+1$ must be odd, we obtain that all the eigenvalues of $\mathbf A$ are positive. Since $\mathbf A$ and ${\mathbf A}^k$ share the same eigenvectors, we get that the corresponding eigenvectors of $\mathbf A$ form a Markov system. Now let us examine the transpose matrix ${\mathbf A}^T$. It is easy to see that ${\mathbf A}^T$ is also eventually strictly totally positive. Applying the same reasoning to ${\mathbf A}^T$ we get that the eigenvectors of ${\mathbf A}^T$ also form a Markov system.
 \qed
\end{proof}
\begin{corollary}
Let an $n\times n$ matrix $\mathbf A$ have a Gantmacher--Krein property (be ESTP). Then ${\mathbf A} + \alpha{\mathbf I}$ also has the Gantmacher--Krein property (respectively, is ESTP), whenever $\alpha > 0$.
\end{corollary}
\begin{proof}
For the proof, it is enough to observe that if $\lambda_1, \ \ldots, \ \lambda_n$ are the eigenvalues of $\mathbf A$ then $\lambda_1 + \alpha, \ \ldots, \ \lambda_n + \alpha$ are the eigenvalues of ${\mathbf A} + \alpha{\mathbf I}$ with the same systems of the corresponding eigenvectors and eigenfunctionals.
\qed
\end{proof}

\begin{example}
 Let us consider the matrix

$${\mathbf A} = \begin{pmatrix} 10 & 2 & 2 \\ 3 & 2 & 1 \\ 7 & 4 & 6  \end{pmatrix}.$$
In this case, we have
$${\mathbf A}^{(2)} = \begin{pmatrix} 14 & 4 & -2 \\ 26 & 46 & 4 \\ -2 & 11 & 8 \end{pmatrix}.$$

$${\mathbf A}^{(3)} = 54.$$

Since $$({\mathbf A}^{(2)})^{3} = \begin{pmatrix} 9980 & 10936 & 40 \\ 80264 & 112156 & 7264 \\ 218400 & 29156 & 2756 \end{pmatrix}$$ and
$$({\mathbf A}^{(2)})^{4} = \begin{pmatrix} 423976 & 543416 & 24104 \\ 4025224 & 5560136 & 346208 \\ 1010144 & 1445092 & 101872 \end{pmatrix}$$ are positive, we apply Theorem 1 and conclude that ${\mathbf A}^{(2)}$ is eventually positive. Thus $\mathbf A$ is ESTP (by Theorem 7).
\end{example}
\begin{example}
Let
$${\mathbf A} = \begin{pmatrix} 8 & 4 & 1 \\ 4 & 10 & 3 \\ -3 & 5 & 9 \end{pmatrix}.$$
In this case we have
$${\mathbf A}^{(2)} = \begin{pmatrix} 64 & 20 & 2 \\ 52 & 75 & 31 \\ 50 & 45 &  75 \end{pmatrix};$$
$${\mathbf A}^{(3)} = \det A = 470.$$

Since ${\mathbf A}^k > 0$ for $k = 5$ and $k = 6$, we apply Theorem 1 and Theorem 7 and obtain that ${\mathbf A}$ is ESTP. However, the eigenvalues of the $2 \times 2$ submatrix $\widehat{{\mathbf A}} = \begin{pmatrix} 8 & 1 \\ -3 & 9 \end{pmatrix}$, obtained from ${\mathbf A}$ by deleting the second row and column, are both complex: $\lambda_1 = \frac{17 + i \sqrt{11}}{2}$ and $\lambda_2 = \frac{17 - i\sqrt{11}}{2}$.
\end{example}

So let us note that {\it a principal submatrix (i.e. obtained from the initial matrix by deleting rows and columns with the same indices) of an ESTP matrix may not be ESTP}.

\section{Similarity transformations preserving the Gantmacher--Krein property and being in GK}
Let us recall some more definitions concerning matrix classes (see \cite{BERPL}, \cite{GANT}).
\begin{definition}
A matrix $\mathbf S$ is called {\it monotone} if it is invertible and ${\mathbf S}^{-1}$ is nonnegative.
\end{definition}

\begin{definition}
Let ${\mathbf A} = \{a_{ij}\}_{i,j = 1}^n$. A matrix $\mathbf A$ is called {\it sign-alternating} if the matrix ${\mathbf A}^*$ with the entries
 $$a_{ij}^* = (-1)^{i+j}a_{ij}, \qquad i,j = 1, \ \ldots, \ n$$
is nonnegative.
\end{definition}

A matrix ${\mathbf A}$ is sign-alternating if and only if it can be written in the following form
$${\mathbf A} = {\mathbf D}\widetilde{{\mathbf A}}{\mathbf D},$$
where $\widetilde{{\mathbf A}}$ is a nonnegative matrix, ${\mathbf D}$ is a diagonal matrix with the diagonal entries $d_{ii} = (-1)^{i+1}$, $i = 1, \ \ldots, \ n$.

\begin{definition}
A real $n \times n$ matrix ${\mathbf A}$ is called {\it totally sign-alternating (TSA)}\footnote{Such matrices are called sign-regular in \cite{GANTKREI}.} if ${\mathbf A}^*$ is TP.
\end{definition}

A matrix ${\mathbf A}$ is TSA if and only if it can be written in the following form
${\mathbf A} = {\mathbf D}\widetilde{{\mathbf A}}{\mathbf D},$
where $\widetilde{{\mathbf A}}$ is a TP matrix, ${\mathbf D}$ is a diagonal matrix with the diagonal entries $d_{ii} = (-1)^{i+1},$ $i = 1, \ \ldots, \ n$.

The following properties of TSA matrices were stated in \cite{GANTKREI}.

\begin{lemma}
Let $\mathbf A$ be an invertible matrix. Then
\begin{enumerate}
\item[\rm 1.] If one of the matrices $\mathbf A$ and ${\mathbf A}^{-1}$ is TP then the other is TSA.
\item[\rm 2.] The matrix $\mathbf A$ is TP if and only if the matrix $({\mathbf A}^*)^{-1}$ is also TP.
\end{enumerate}
\end{lemma}

The similarity matrices preserving the strong Perron--Frobenius property and the class PF are described in \cite{ELH}. The following statements are proved in \cite{ELH} (see \cite{ELH}, p. 41, Theorems 3.6 and 3.7).

\begin{theorem}
For any invertible matrix $\mathbf S$, the following statements are equivalent:
\begin{enumerate}
\item[\rm 1.] Either $\mathbf S$ or $-{\mathbf S}$ is monotone.
\item[\rm 2.] $\mathbf {\mathbf S}^{-1}{\mathbf A}{\mathbf S}$ has the strong Perron--Frobenius property for all matrices $\mathbf A$ having the strong Perron--Frobenius property.
\end{enumerate}
\end{theorem}

\begin{theorem}
For any invertible matrix $\mathbf S$, the following statements are equivalent:
\begin{enumerate}
\item[\rm 1.]  $\mathbf S$ and ${\mathbf S}^{-1}$ are either both nonnegative or both nonpositive.
\item[\rm 2.] $\mathbf {\mathbf S}^{-1}{\mathbf A}{\mathbf S} \in PF$ for all matrices ${\mathbf A} \in PF$.
\end{enumerate}
\end{theorem}

Now we are going to analyze which similarity matrices $\mathbf S$ preserve the Gantmacher--Krein property or being in GK.

\begin{theorem}
For any invertible matrix $\mathbf S$, the following statements are equivalent:
\begin{enumerate}
\item[\rm 1.] Either $\mathbf S$ or $-{\mathbf S}$ is TSA.
\item[\rm 2.] $\mathbf {\mathbf S}^{-1}{\mathbf A}{\mathbf S}$ has the Gantmacher--Krein property for all matrices $\mathbf A$ having the \linebreak  Gantmacher--Krein property.
\end{enumerate}
\end{theorem}

\begin{proof} $(1) \Rightarrow (2).$ Suppose (1) holds. Assume without loss of generality that $\mathbf S$ is TSA. Then ${\mathbf S}^{-1}$ is TP and $({\mathbf S}^{-1})^{(j)}$ is nonnegative for all $j = 1, \ \ldots, \ n$. The Jacobi formula $({\mathbf S}^{-1})^{(j)} = ({\mathbf S}^{(j)})^{-1}$ shows that ${\mathbf S}^{(j)}$ is monotone for $j = 1, \ \ldots, \ n$.

Let $\mathbf A$ be an arbitrary $n \times n$ matrix having the Gantmacher--Krein property. Applying Lemma 4 to $\mathbf A$, we obtain that the $j$th compound matrix ${\mathbf A}^{(j)}$ has the strong Perron--Frobenius property for all $j = 1, \ \ldots, \ n$. Let us examine the matrix ${\mathbf S}^{-1}{\mathbf A}{\mathbf S}$. The Cauchy--Binet formula implies the equality
$$({\mathbf S}^{-1}{\mathbf A}{\mathbf S})^{(j)} = ({\mathbf S}^{-1})^{(j)}{\mathbf A}^{(j)}{\mathbf S}^{(j)}.$$
Applying Theorem 8 to every ${\mathbf A}^{(j)}$, we obtain that $({\mathbf S}^{-1})^{(j)}{\mathbf A}^{(j)}{\mathbf S}^{(j)}$ has the strong Perron--Frobenius property for all $j = 1, \ \ldots, \ n$. Applying Lemma 4 to ${\mathbf S}^{-1}{\mathbf A}{\mathbf S}$, we complete the proof.

$(2) \Rightarrow (1).$ Conversely, suppose (1) does not hold, i.e. ${\mathbf S}$ and $-{\mathbf S}$ are both not TSA. Then there is a positive integer $j$, $1 \leq j \leq n$ such that the $j$th compound matrix $({\mathbf S}^{-1})^{(j)} = ({\mathbf S}^{(j)})^{-1}$ has a positive entry and a negative entry.

In this case, following the reasoning from \cite{ELH}, we can find a positive vector $v$ such that ${({\mathbf S}^{(j)})^{-1}}v$ has a positive entry and a negative entry. Here we consider the following two cases.
\begin{enumerate}
\item[\rm (a)] The matrix ${({\mathbf S}^{(j)})^{-1}}$ has a column (say, $l$th column) with a positive entry and a negative entry. In this case, we put $v \in {\mathbb R}^{\binom{n}j}$ with the coordinates $v = (v^1, \ \ldots, \ v^{\binom{n}j})$, where $v^l = 1$, $v^i = \epsilon_i > 0$, $i \neq l$, $1 \leq i \leq \binom{n}j$. In this case, ${({\mathbf S}^{(j)})^{-1}}v$ has a positive entry and a negative entry, for sufficiently small values $\epsilon_i$.
\item[\rm (b)] Every nonzero column of ${({\mathbf S}^{(j)})^{-1}}$ is either nonpositive or nonnegative (with at least one nonzero entry). Let us assume that the $l$th column is nonnegative and the $m$th column is nonpositive. Without loss of generality, we assume that $l$ and $m$ are the numbers in the lexicographic numeration of the sets of indices $(i_1, \ \ldots, \ i_{r-1}, \ i_{r+1}, \ \ldots, \ i_{j+1})$ and $(i_1, \ \ldots, \ i_{s-1}, \ i_{s+1}, \ \ldots, \ i_{j+1})$, respectively (here $1 \leq i_1 < \ldots < i_{j+1} < n$, $1 \leq r, \ s \leq j+1$, $r \neq s$). (Indeed, suppose that all nonzero entries of each two columns of ${({\mathbf S}^{(j)})^{-1}}$ with the numbers as above are of the same sign (say, positive). In this case it is easy to see that the whole matrix ${({\mathbf S}^{(j)})^{-1}}$ is nonnegative.) Let us consider ${({\mathbf S}^{(j)})^{-1}} ((1 - \lambda)\widetilde{e}_l + \lambda\widetilde{e}_m)$, where $\widetilde{e}_l, \ \widetilde{e}_m$ are the $l$th and the $m$th basic vectors in ${\mathbb R}^{\binom{n}j}$ respectively, $\lambda \in [0,1]$. Note that ${({\mathbf S}^{(j)})^{-1}}\widetilde{e}_l$ is nonnegative and ${({\mathbf S}^{(j)})^{-1}}\widetilde{e}_m$ is nonpositive. Let $\lambda_0$ be the largest number from $[0,1]$ such that ${({\mathbf S}^{(j)})^{-1}} ((1 - \lambda)\widetilde{e}_l + \lambda\widetilde{e}_m)$ is still nonnegative. Since all the columns of ${({\mathbf S}^{(j)})^{-1}}$ are linearly independent, we obtain that the vector ${({\mathbf S}^{(j)})^{-1}} ((1 - \lambda_0)\widetilde{e}_l + \lambda_0\widetilde{e}_m)$ is nonzero for any $\lambda_0$. Let us choose $\lambda_1 > \lambda_0$, sufficiently close to $\lambda_0$. Then ${({\mathbf S}^{(j)})^{-1}} ((1 - \lambda_1)\widetilde{e}_l + \lambda_1\widetilde{e}_m)$ has a positive entry and a negative entry. Now let $v = (v^1, \ \ldots, \ v^{\binom{n}j})$ be the positive vector in ${\mathbb R}^{\binom{n}j}$ with $v^l = 1 - \lambda_1$, $v^m = \lambda_1$ and $v^i = \epsilon_i$, $1 \leq i \leq {\binom{n}j},$ $i \neq l, \ m$. Then ${({\mathbf S}^{(j)})^{-1}}v$ has a positive entry and a negative entry, for sufficiently small $\epsilon_i$.
\end{enumerate}
Now let us fix an arbitrary STP matrix $\mathbf A$. By Gantmacher--Krein theorem (Theorem 5), $\mathbf A$ has the Gantmacher--Krein property. Thus, applying Lemma 4, we obtain that the $j$th compound matrix ${\mathbf A}^{(j)}$ has the strong Perron--Frobenius property, that is, the first eigenvector $\varphi_j = (\varphi_j^1, \ \ldots, \ \varphi_j^{\binom{n}j})$ corresponding to the greatest in absolute value eigenvalue $\rho({\mathbf A}^{(j)})$ may be chosen to be positive. Let us construct a positive diagonal matrix $\mathbf D$ as follows:
\begin{enumerate}
\item[\rm (a')] For the case (a), let $l$ be the number in the lexicographic numeration of the set of indices $(i_1, \ \ldots, \ i_j)$, $1 \leq i_1 < \ldots < i_j \leq n$. Then we put ${\mathbf D} = {\rm diag}\{d_{11},  \ \ldots, \ d_{nn}\}$, where
    $$d_{kk} = \left\{\begin{array}{cc} \epsilon^{j-1}\varphi_j^l, & \mbox{if} \ k = i_1;
\\[10pt] \frac{1}{\epsilon}, & \mbox{if} \ k \in \{i_2, \ \ldots, \ i_j\};
\\[10pt] \frac{\max_p\varphi_j^p}{\epsilon^2} & \mbox{if} \ k \in [n]\setminus\{i_1, \ i_2, \ \ldots, \ i_j\}.
\end{array}\right.$$
\item[\rm (b')] For the case (b), where $l$ and $m$ are the numbers in the lexicographic numeration of the sets of indices $(i_1, \ \ldots, \ i_{r-1}, \ i_{r+1}, \ \ldots, \ i_{j+1})$ and $(i_1, \ \ldots, \ i_{s-1}, \ i_{s+1}, \ \ldots, \ i_{j+1})$, respectively (here $1 \leq i_1 < \ldots < i_{j+1} < n$, $1 \leq r, \ s \leq j+1$, $r \neq s$), we put ${\mathbf D} = {\rm diag}\{d_{11},  \ \ldots, \ d_{nn}\}$ as follows:
     $$d_{kk} = \left\{\begin{array}{cc} \frac{1}{(1-\lambda_1)\epsilon^{j-1}}\varphi_j^l, & \mbox{if} \ k = i_s;
     \\[10pt] \frac{1}{\lambda_1\epsilon^{j-1}}\varphi_j^m,, & \mbox{if} \ k = i_r;
\\[10pt] \epsilon, & \mbox{if} \ k \in \{i_1, \ \ldots, \ i_{j+1}\}\setminus\{i_r, \ i_s\};
\\[10pt] \frac{\max_p\varphi_j^p}{\epsilon^j} & \mbox{if} \ k \in [n]\setminus\{i_1, \ i_2, \ \ldots, \ i_{j+1}\}.
\end{array}\right.$$
\end{enumerate}
Considering the matrix ${\mathbf B} = {\mathbf D}^{-1}{\mathbf A}{\mathbf D}$, we obtain by the Cauchy--Binet formula that ${\mathbf B}$ is also STP, for any positive diagonal matrix $\mathbf D$. Then, applying the Gantmacher--Krein theorem (Theorem 5), we obtain that ${\mathbf B}$ has the Gantmacher--Krein property, and, by Lemma 4, ${\mathbf B}^{(j)}$ has the strong Perron--Frobenius property. Now us consider the matrix ${\mathbf S}^{-1}{\mathbf B}{\mathbf S}$ and its $j$th compound matrix $({\mathbf S}^{-1}{\mathbf B}{\mathbf S})^{(j)} = ({\mathbf S}^{(j)})^{-1}{\mathbf B}^{(j)}{\mathbf S}^{(j)}$ (through the Cauchy--Binet and Jacobi formulae). It is easy to see, that the eigenvalues of $({\mathbf S}^{(j)})^{-1}{\mathbf B}^{(j)}{\mathbf S}^{(j)}$ are the same that those of ${\mathbf B}^{(j)}$ and the first eigenvector of \linebreak $({\mathbf S}^{(j)})^{-1}{\mathbf B}^{(j)}{\mathbf S}^{(j)}$, corresponding to the greatest in absolute value eigenvalue, is of the form $({\mathbf S}^{(j)})^{-1}\psi_j$, where $\psi_j$ is the first eigenvector of ${\mathbf B}^{(j)}$ (corresponding to $\rho({\mathbf B}^{(j)})$). However, since ${\mathbf B} = {\mathbf D}^{-1}{\mathbf A}{\mathbf D}$, we have the equality ${\mathbf B}^{(j)} = ({\mathbf D}^{(j)})^{-1}{\mathbf A}^{(j)}{\mathbf D}^{(j)}$ (through the Cauchy--Binet and Jacobi formulae). Thus the first eigenvector $\psi_j$ of ${\mathbf B}^{(j)}$ is of the form $({\mathbf D}^{(j)})^{-1}\varphi_j$, where $\varphi_j$ is the first eigenvector of ${\mathbf A}^{(j)}$ corresponding to the greatest in absolute value eigenvalue $\rho({\mathbf A}^{(j)})$. Considering the entries of ${\mathbf D}^{-1}$, we obtain the following equalities:
$$\psi_j = ({\mathbf D}^{(j)})^{-1}\varphi_j = v;$$
So we obtain the contradiction: $({\mathbf S}^{(j)})^{-1}\psi_j = ({\mathbf S}^{(j)})^{-1}v$ by Lemma 4, must be positive, but, in the same time, it is chosen to have both positive entries and negative entries.
 \qed
\end{proof}

\begin{theorem}
For any $n \times n$ invertible matrix $\mathbf S$, the following statements are equivalent:
\begin{enumerate}
\item[\rm 1.]  $\mathbf S$ is positive (negative) diagonal matrix.
\item[\rm 2.] $\mathbf {\mathbf S}^{-1}{\mathbf A}{\mathbf S} \in GK$ for all matrices ${\mathbf A} \in GK$.
\end{enumerate}
\end{theorem}

\begin{proof} ${\rm (1)} \Rightarrow {\rm (2)}$. Let $\mathbf A$ be an arbitrary matrix from GK. Applying Theorem 7 to $\mathbf A$, we obtain that the $j$th compound matrix ${\mathbf A}^{(j)}$ belongs to PF for every $j = 1, \ \ldots, \ n$. Let $\mathbf S$ be a positive diagonal matrix (otherwise we'll consider $- {\mathbf S}$). In this case, it is easy to see that ${\mathbf S}^{(j)}$ is also a positive diagonal matrix for every $j = 1, \ \ldots, \ n$. Applying the Cauchy--Binet and Jacobi formulae, we obtain the equality $$({\mathbf S}^{-1}{\mathbf A}{\mathbf S})^{(j)} = ({\mathbf S}^{(j)})^{-1}{\mathbf A}^{(j)}{\mathbf S}^{(j)}. \qquad (j = 1, \ \ldots, \ n)$$
Since ${\mathbf S}^{(j)}$ and $({\mathbf S}^{(j)})^{-1}$ are both nonnegative, we apply Theorem 9 and obtain that the matrix $({\mathbf S}^{(j)})^{-1}{\mathbf A}^{(j)}{\mathbf S}^{(j)}$ also belongs to PF for every $j = 1, \ \ldots, \ n$. Applying Theorem 7 to the matrix ${\mathbf S}^{-1}{\mathbf A}{\mathbf S}$ we obtain that ${\mathbf S}^{-1}{\mathbf A}{\mathbf S}$ belongs to GK.

${\rm (2)} \Rightarrow {\rm (1)}$. It is enough to prove that if $\mathbf S$ and ${\mathbf S}^{-1}$ are both TP then $\mathbf S$ is a positive diagonal matrix. If ${\mathbf S}^{-1}$ is TP then using Lemma 5 we obtain that $\mathbf S$ is TSA. Thus we have that the matrices $\mathbf S$ and ${\mathbf S}^*$ are both TP. Let us prove that it is possible only if $\mathbf S$ is positive diagonal. We will prove this by induction on $n$.

For $n=2$, we have ${\mathbf S} = \begin{pmatrix} s_{11} & s_{12} \\ s_{21} & s_{22} \end{pmatrix} \geq 0$, ${\mathbf S}^* = \begin{pmatrix} s_{11} & -s_{12} \\ -s_{21} & s_{22} \end{pmatrix} \geq 0$ and $\det ({\mathbf S}) > 0$. Obviously, this is possible if and only if $s_{12} = s_{21} = 0$, $s_{11}, \ s_{22} > 0$. For $n = 2$, the statement holds. Let it holds for $n-1$. We prove it for $n$.
We have that an $n \times n$ matrices ${\mathbf S}$ and ${\mathbf S}^{*}$ are both TP. Let us consider $\widetilde{{\mathbf S}_1}$ and $\widetilde{{\mathbf S}_n}$ --- two $(n-1) \times (n-1)$ principal submatrices of ${\mathbf S}$, obtained by deleting the first (respectively, the last) row and column. The following equalities hold for these submatrices:
$$(\widetilde{{\mathbf S}_n})^* = \widetilde{{\mathbf S}^*}_n ;$$
$$(\widetilde{{\mathbf S}_1})^* = \widetilde{{\mathbf S}^*}_1.$$

This equalities and total positivity of ${\mathbf S}$ and ${\mathbf S}^{*}$ imply that all the matrices $\widetilde{{\mathbf S}_n}$, $\widetilde{{\mathbf S}_1}$, $(\widetilde{{\mathbf S}_n})^*$ and $(\widetilde{{\mathbf S}_1})^*$ are totally positive. Thus we can apply the induction hypothesis and obtain that $\widetilde{{\mathbf S}_n}$ and $\widetilde{{\mathbf S}_1}$ are both positive diagonal. This implies $s_{ii} >0$ for all $i = 1, \ \ldots, \ n$ and all off-diagonal entries of ${\mathbf S}$, except, probably, $s_{1n}$ and $s_{n1}$ are equal to zero. Let us prove that $s_{1n}$ and $s_{n1}$ are also equal to zero. If $n$ is even then $s^*_{1n} = - s_{1n}$, $s^*_{n1} = - s_{n1}$ and the inequalities $s^*_{1n}, \ s^*_{n1}, \ s_{1n}, \ s_{n1} \geq 0$ imply $s_{n1} = s_{1n} = 0$. If $n$ is odd, let us consider the minors $S \begin{pmatrix}1 & 2 \\ 2 & n \end{pmatrix}$ and $S \begin{pmatrix}n-1 & n \\ 1 & n-1 \end{pmatrix}$. If $s_{n1}> 0$ or $s_{1n} > 0$ then we have one of the following estimates
$$S \begin{pmatrix}1 & 2 \\ 2 & n \end{pmatrix} = s_{12}s_{2n} - s_{22}s_{1n} = - s_{22}s_{1n} < 0;$$
$$S \begin{pmatrix}n-1 & n \\ 1 & n-1 \end{pmatrix} = s_{n-1, 1}s_{n,n-1} - s_{n-1,n-1}s_{n1} = - s_{n-1,n-1}s_{n1} < 0.$$
This contradicts the nonnegativity of ${\mathbf S}^{(2)}$. \qed
\end{proof}
\section{Eventually STJS matrices}
Now let us give the following definitions.
\begin{definition}
An $n\times n$ matrix $\mathbf A$ is said to have the {\it total signature equality (TSE) property} if $\mathbf A$ has $n$ positive simple eigenvalues $\{\lambda_1, \ \ldots, \ \lambda_n\}$ with the systems of the corresponding eigenvectors $\{x_1, \ \ldots, \ x_n\}$ and eigenfunctionals $\{x_1^*, \ \ldots, \ x_n^*\}$ satisfying the following conditions:
\begin{enumerate}
\item Both $x_1 \wedge \ldots \wedge x_j$ and $x_1^* \wedge \ldots \wedge x_j^*$ have no zero coordinates for all $j = 1, \ \ldots, \ n$.
\item ${\rm Sign}(x_1 \wedge \ldots \wedge x_j) = {\rm Sign}(x_1^* \wedge \ldots \wedge x_j^*)$ for all $j = 1, \ \ldots, \ n$.
\end{enumerate}
\end{definition}

It is obvious that if $\mathbf A$ has TSE property then so does ${\mathbf A}^T$.

Now let us present the following generalizations of total positivity (for the definition and examples, see also \cite{KU}).

\begin{definition}
A $n \times n$ matrix ${\mathbf A}$ is called {\it totally J--sign-symmetric (TJS)}, if it is J--sign-symmetric, and its $j$-th compound matrices ${\mathbf A}^{(j)}$ are also J--sign-symmetric for every $j$ $(j = 2, \ \ldots, \ n)$.
\end{definition}

\begin{definition}
A $n \times n$ matrix ${\mathbf A}$ is called {\it strictly totally J--sign-symmetric (STJS)}, if it is strictly J--sign-symmetric, and its $j$-th compound matrices ${\mathbf A}^{(j)}$ are also strictly J--sign-symmetric for every $j$ $(j = 2, \ \ldots, \ n)$.
\end{definition}

Let us give an example of an STJS matrix.

\begin{example}
Let us take
$${\mathbf A} = \begin{pmatrix} 5.6 & 1.2 & 0.7 & 0.5 \\ 6.6 & 6.2 & 4.1 & 8.1 \\ 4.4 & 4.4 & 3.5 & 8\\ 1 & 3.8 & 3.4 & 9 \end{pmatrix}$$
In this case, we have
$${\mathbf A}^{(2)} = \begin{pmatrix} 26.8 & 18.34 & 42.06 & 0.58 & 6.62 & 3.62 \\
 19.36 & 16.52 & 42.6 & 1.12 & 7.4 & 3.85 \\
 20.08 & 18.34 & 49.9 & 1.42 & 8.9 & 4.6 \\
 1.76 & 5.06 & 17.16 & 3.66 & 13.96 & 4.45 \\
 18.88 & 18.34 & 51.3 & 5.5 & 25.02 & 9.36 \\
 12.32 & 11.46 & 31.6 & 1.66 & 9.2 & 4.3 \\
\end{pmatrix};$$

$${\mathbf A}^{(3)} = \begin{pmatrix}
 15.656 & 58.464 & 15.438 & -2.602 \\
 22.008 & 87.992 & 25.676 & -3.532 \\
 4.168 & 19.76 & 7.69 & -0.45 \\
 -9.584 & -35.408 & -8.354 & 2.386 \\
\end{pmatrix};$$

$${\mathbf A}^{(4)} = \det({\mathbf A}) = 3.3928.$$
\end{example}

The following statement describes the spectral properties of STJS matrices (see \cite{KU}, p. 559, Theorem 31).

\begin{theorem} Let an $n \times n$ matrix ${\mathbf A}$ be
STJS. Then
all the eigenvalues of the operator $A$ are positive and simple:
$$\rho(A) = \lambda_1 > \lambda_2 > \ldots > \lambda_n > 0.$$
\end{theorem}

If $\mathbf A$ is STJS then ${\mathbf A}^T$ is also STJS (see \cite{KU}, p.558, Proposition 29). In this case, we have ${\rm Sign}({\mathbf A}^{(j)}) = {\rm Sign}(({\mathbf A}^{(j)})^T)$ for $j = 1, \ \ldots, \ n$ and it is not difficult to see that {\it an STJS matrix $\mathbf A$ has the total signature equality property}.

Analogically with ESTP matrices, we introduce the following generalization of the class of STJS matrices.

\begin{definition}
A real $n \times n$ matrix ${\mathbf A}$ is called {\it eventually strictly totally J-sign-symmetric
(ESTJS)} if there is a positive integer $k_0$ such that  ${\mathbf A}^k$ is STJS for all
$k \geq k_0$.
\end{definition}

The following result characterizes properties of ESTJS matrices.

\begin{theorem}
Let ${\mathbf A}$ be an $n \times n$ matrix. Then the following statements are equivalent.
\begin{enumerate}
\item[\rm 1.] The matrix ${\mathbf A}$ has the total signature equality property.
\item[\rm 2.] For every $j$, $j = 1, \ \ldots, \ n$, the $j$th compound matrix ${\mathbf A}^{(j)}$ has the signature equality property.
\item[\rm 3.] For every $j$, $j = 1, \ \ldots, \ n$, the $j$th compound matrix ${\mathbf A}^{(j)}$ is ESJS.
\item[\rm 4.] The matrix ${\mathbf A}$ is eventually STJS.
\end{enumerate}
\end{theorem}
\begin{proof}
$(1) \Rightarrow (2)$. The proof just follows from the definition and the Kronecker theorem (Theorem 4), applied to $\mathbf A$ and ${\mathbf A}^T$.

$(2) \Rightarrow (3)$. The proof follows from Theorem 3, applied to each ${\mathbf A}^{(j)}$, $j = 1, \ \ldots, \ n$.

$(3) \Rightarrow (4)$. We just copy the reasoning of the proof of Theorem 7, implication $(3) \Rightarrow (4)$.

$(4) \Rightarrow (1)$. We just copy the reasoning of the proof of Theorem 7, implication $(4) \Rightarrow (1)$, replacing the Gantmacher--Krein theorem (Theorem 5) with Theorem 12.
\qed
\end{proof}
\section{Eventually $P$-matrices}
Let us recall the following definition (see \cite{FIP}).
\begin{definition}
An $n \times n$ matrix $\mathbf A$ is called a {\it $P$-matrix} if all its principal minors are positive, i.e the inequality $A \begin{pmatrix}i_1 & \ldots & i_k \\ i_1 & \ldots & i_k \end{pmatrix} > 0$
holds for all sets of indices $(i_1, \ \ldots, \ i_k), \ 1 \leq i_1 < \ldots < i_k \leq n$, and all $k, \ 1 \leq k \leq n$.
\end{definition}

The corresponding eventual property of matrices is defined as follows.

\begin{definition}
A matrix ${\mathbf A}$ is called an eventually {\it $P$-matrix} if there is a positive integer $k_0$ such that ${\mathbf A}^k$ is a $P$-matrix for all positive integers $k \geq k_0$.
\end{definition}

Assuming the positivity and simplicity of the spectrum, we obtain the following result describing the structure of eventually $P$-matrices.

\begin{theorem}
Let $\mathbf A$ be an $n \times n$ eventually $P$-matrix whose eigenvalues are all positive, simple and different in absolute value from each other:
$$ \rho(A) = \lambda_1 > \lambda_2 > \ldots > \lambda_n.$$
Then $\mathbf A$ is ESTJS.
\end{theorem}
\begin{proof} Let us consider the $j$th compound matrix ${\mathbf A}^{(j)}$, $j = 1, \ \ldots, \ n$. Applying the Kronecker theorem (Theorem 4) to ${\mathbf A}^{(j)}$ we obtain that $\rho({\mathbf A}^{(j)}) = \lambda_1 \ldots \lambda_j$ is a positive simple strictly dominant eigenvalue of ${\mathbf A}^{(j)}$. Thus the conditions of Lemma 1 hold. Applying Lemma 1 to each ${\mathbf A}^{(j)}$, $j = 1, \ \ldots, \ n$, we obtain the approximation:
$$\frac{1}{\rho({\mathbf A}^{(j)})^k}({\mathbf A}^{(j)})^k \rightarrow \varphi_j \otimes \varphi^*_j \qquad \mbox{as} \quad k \rightarrow \infty,$$
where $\varphi_j = (\varphi_j^1, \ \ldots, \ \varphi_j^{\binom{n}j})$ and $\varphi_j^*= ((\varphi_j^*)^1, \ \ldots, \ (\varphi_j^*)^{\binom{n}j})$ are the eigenvector and the eigenfunctional of ${\mathbf A}^{(j)}$ corresponding to $\rho({\mathbf A}^{(j)})$, respectively. Applying the Cauchy--Binet formula, we obtain that
$$({\mathbf A}^{(j)})^k = ({\mathbf A}^{k})^{(j)}.$$ Since $\mathbf A$ is an eventually $P$-matrix, we have that $({\mathbf A}^{(j)})^k$ has positive principal diagonal entries for sufficiently big $k$.

Let us consider the principal diagonal entries of the matrix $\varphi_j \otimes \varphi^*_j$. These are $\varphi_j^i(\varphi_j^*)^i$, $i = 1, \ \ldots, \ \binom{n}j.$ So the following inequalities hold:
$$\varphi_j^i(\varphi_j^*)^i > 0, \qquad  i = 1, \ \ldots, \ \binom{n}j.$$
It follows that ${\rm Sign}(\varphi_j) = {\rm Sign}(\varphi_j^*)$, i.e. ${\mathbf A}^{(j)}$ has the signature equality property. Applying Theorem 3, we get that ${\mathbf A}^{(j)}$ is eventually SJS for all $j = 1, \ \ldots, \ n$. Then, applying Theorem 13 we get that $\mathbf A$ is eventually STJS.
\qed
\end{proof}
\begin{corollary} Any eventually $P$-matrix with a positive simple distinct spectrum has the total signature equality property.
\end{corollary}
\begin{proof} Follows from the above reasoning and Theorem 13.
\qed
\end{proof}



\begin{thebibliography}{30}
\bibitem{AN}
{\sc T. Ando}, {\it Totally positive matrices.} Linear Algebra Appl. {\bf 90} (1987), 165-219.

\bibitem{BE}
{\sc A. Berman, M. Catral , L.M. Dealba , A. Elhashash, F.J. Hall, L. Hogben , I. Kim , D.D. Olesky, P. Tarazaga , M.J. Tsatsomeros and P. Van Den Driessche},
{\it Sign patterns that allow eventual positivity}, ELA {\bf 19} (2010), 108-120.

\bibitem{BERPL}
{\sc A. Berman and R.J. Plemmons}, {\it Nonnegative Matrices in the
Mathematical Sciences.} Academic Press, New York, 1979.

\bibitem{ELH}
{\sc A. Elhashash}, {\it Characterizations of matrices enjoying the Perron-Frobenius property and generalizations of M-matrices which may not have nonnegative inverses}, Ph.D., Temple University, 2008.

\bibitem{ELS1}
{\sc A. Elhashash and D.B. Szyld}, {\it Two characterizations of matrices with the Perron--Frobenius property,} Numer. Linear Algebra Appl. {\bf 16} (2009), 863-869.

\bibitem{ELS2}
{\sc A. Elhashash and D.B. Szyld}, {\it On general matrices having the Perron--Frobenius property}, ELA {\bf 17} (2008), 389-413.

\bibitem{EHT}
{\sc E.M. Ellison, L. Hogben and M.J. Tsatsomeros}, {\it Sign patterns that require eventual positivity or require eventual nonnegativity}, ELA {\bf 19} (2009-2010), 98-107.

\bibitem{FAJ}
{\sc Sh. Fallat and C.R. Johnson}, {\it Totally nonnegative matrices}, Princ. Univ. Press, 2011.

\bibitem{FIP}
{\sc M. Fiedler and V. Pt\'{a}k}, {\it On matrices with non-positive off-diagonal elements and positive principal minors}, Czech. Math. J. {\bf 22 (87)} (1962), 382-400.

\bibitem{FRIED2}
{\sc S. Friedland}, {\it On an inverse problem for nonnegative and eventually nannegative matrices}, Israel Journal of Mathematics {\bf 29} (1978), 43-60.

\bibitem{GANT}
{\sc F. Gantmacher}, {\it The Theory of Matrices,} Volume 1,
Volume 2. Chelsea. Publ. New York, 1990.

\bibitem{GANTKREI}
{\sc F.R. Gantmacher and M.G. Krein},
{\it Oscillation Matrices and Kernels and Small Vibrations of Mechanical Systems,} AMS
Bookstore, 2002.

\bibitem{GLALU}
{\sc I.M. Glazman and Yu.I. Liubich}, {\it Finite-Dimensional Linear Analysis: A Systematic Presentation in Problem Form,} MIT Press, 1974.

\bibitem{JT}
{\sc C. R. Johnson and P. Tarazaga}, {\it On Matrices with Perron–Frobenius Properties and Some Negative Entries,} Positivity {\bf 8} (2004), 327-338.

\bibitem{KU}
{\sc O.Y. Kushel}, {\it Cone-theoretic generalization of total positivity}, Linear Algebra Appl. {\bf 436} (2012), 537-560.

\bibitem{NO}
{\sc D. Noutsos}, {\it On Perron--Frobenius property of matrices having some negative entries}, Linear Algebra Appl. {\bf 412} (2006), 132--153.

\bibitem{PINK}
{\sc A. Pinkus}, {\it Totally positive matrices,} Cambridge University
Press, 2010.

\bibitem{TARH}
{\sc P. Tarazaga, M. Raydan and A. Hurman}, {\it Perron--Frobenius theorem for matrices with some negative entries,} Linear Algebra Appl. {\bf 328} (2001), 57--68.

\end{thebibliography}
\end{document}